\documentclass[a4paper,11pt]{article}

\pdfoutput=1

\usepackage[utf8]{inputenc}
\usepackage{fullpage}
\usepackage{amsmath}
\usepackage{amsthm}
\numberwithin{equation}{section}
\usepackage{amssymb}
\usepackage{hyperref}
\usepackage{bbm}
\usepackage{graphicx}

\newtheoremstyle{customDef}
  {2.0\topsep}   
  {\topsep}   
  {\normalfont}  
  {0pt}       
  {\bfseries} 
  {.}         
  {5pt plus 1pt minus 1pt} 
  {}          
  
\newtheoremstyle{customThm}
{2.0\topsep}  
{\topsep}     
{\itshape}     
{0pt}         
{\bfseries}   
{.}           
{5pt plus 1pt minus 1pt} 
{}            

\theoremstyle{customDef}

\newtheorem*{dfn*}{Definition}

\newtheorem*{corol*}{Corollary}

\theoremstyle{customThm}

\newtheorem{thm}{Theorem}[section]
\newtheorem*{thm*}{Theorem}
\newtheorem{lem}{Lemma}[section]

\newenvironment{enumerate*}
{	\begin{enumerate} 
	\setlength{\itemsep}{0pt}
}
{\end{enumerate}}

\usepackage{xifthen}
\newcommand{\norm}[1]{\ifthenelse{\isempty{#1}}{\left\| \cdot \right\|}{\left\| #1 \right\|}}

\newcommand{\R}{\mathbb{R}}

\newcommand{\twiddle}{\sim}

\renewcommand{\hat}[1]{\widehat{#1}}

\usepackage{caption}
\usepackage{wrapfig}

\newcommand\blfootnote[1]{%
  \begingroup
  \renewcommand\thefootnote{}\footnote{#1}%
  \addtocounter{footnote}{-1}%
  \endgroup
}
\begin{document}

\title{Clouds in Higher Dimensions}
\author{Samuel Desrochers}
\maketitle

\begin{abstract}
Following some work done by Komjáth \cite{ThreeClouds} and Schmerl \cite{ConversenClouds}, we extend the definition of a cloud to $\R^N$ for $N \geq 2$ and show that $k$ clouds cover $\R^2$ if and only if $k$ clouds cover $\R^N$. We also show that countably many clouds cover $\R^N$.
\end{abstract}

\section{Introduction and Main Results}
\label{sect:intro}

In this paper, we are concerned with certain subsets of Euclidean space called clouds. These sets were first described by Komjáth \cite{ThreeClouds}, who gave the following definition:

\begin{quote}
If $a$ is a point on the plane, then a \textit{cloud around $a$} is a set $A$ which intersects every line $e$ with $a \in e$ in a finite set.
\end{quote}

Komjáth's main interest was determining the number of clouds needed to cover the plane. He showed that three clouds cover the plane if and only if the continuum hypothesis holds. He was also able to generalize this statement and show that if $2^{\aleph_0} \leq \aleph_n$, then $n+2$ clouds cover the plane ($n \geq 1$). However, whether the converse of this statement was true remained open; this was only resolved later by Schmerl \cite{ConversenClouds}. In fact, joining their results yields the following, slightly stronger, statement.

\begin{thm}[Komjáth, Schmerl]
\label{TFAE}
Let $n \geq 1$. The following are equivalent:
\begin{enumerate*}
\item $2^{\aleph_0} \leq \aleph_n$
\item $n+2$ clouds cover $\R^2$
\item For any $n+2$ distinct noncollinear points in $\R^2$, there are clouds centered at these points which cover $\R^2$
\end{enumerate*}
\end{thm}

In this paper, we are interested in extending the notion of a cloud to higher-dimensional Euclidean spaces, and determining the number of clouds required to cover these spaces. Firstly, adapting the definition of a cloud to higher dimensions is easy: indeed, there is almost nothing to change.

\begin{dfn*}
Let $N \geq 2$ and $a \in \R^N$. A \textbf{cloud around $a$} is a subset $C \subset \R^N$ such that for every line $L$ passing through $a$, $L \cap C$ is finite. More generally, we say $C \subset \R^N$ is a \textbf{cloud} if it is a cloud around $a$ for some $a \in \R^N$.
\end{dfn*}

We now wish to determine the number of clouds needed to cover $\R^N$ for any $N \geq 2$. In particular, our main result consists in establishing the relationship between the number of clouds needed to cover $\R^2$ and the number of clouds needed to cover $\R^N$ for any $N$. Without assuming anything about the size of the continuum, we obtain the following result.

\begin{thm}
\label{thm:2iffN}
Let $N \geq 2, \; n \geq 3$. Then $n$ clouds cover $\R^2$ if and only if $n$ clouds cover $\R^N$.
\end{thm}

An immediate corollary of this theorem is:

\begin{corol*}
Let $N,M \geq 2$, $n \geq 3$. Then $n$ clouds cover $\R^N$ if and only if $n$ clouds cover $\R^M$.
\end{corol*}

More interesting is the way in which we obtain a proof of this result. To obtain a covering of $\R^N$ from a covering of $\R^2$, all we need is a straightforward geometric argument. However, to prove the reverse implication, we must adapt Schmerl's argument to higher dimensions, so that a covering of $\R^N$ implies the continuum hypothesis. Then theorem \ref{TFAE} indirectly gives the implication we need. To make all of this more explicit, we claim that the following extension of theorem \ref{TFAE} is true.

\begin{thm}
\label{NewTFAE}
Let $n \geq 1$ and $N \geq 2$. The following are equivalent:
\begin{enumerate*}
\item \label{NewTFAE:CH} 
$2^{\aleph_0} \leq \aleph_n$
\item \label{NewTFAE:2} 
$n+2$ clouds cover $\R^2$
\item \label{NewTFAE:2p}
For any $n+2$ distinct noncollinear points in $\R^2$, there are clouds centered at these points which cover $\R^2$
\item \label{NewTFAE:N}
$n+2$ clouds cover $\R^N$
\item \label{NewTFAE:Np}
For any $n+2$ distinct noncollinear points in $\R^N$, there are clouds centered at these points which cover $\R^N$
\end{enumerate*}
\end{thm}

Note that the implication $\ref{NewTFAE:Np} \Rightarrow \ref{NewTFAE:N}$ is trivial. Then to prove this equivalence, it will suffice to prove the implications $\ref{NewTFAE:2p} \Rightarrow \ref{NewTFAE:Np}$ and $\ref{NewTFAE:N} \Rightarrow \ref{NewTFAE:CH}$. As we noted previously, the implication $\ref{NewTFAE:2p} \Rightarrow \ref{NewTFAE:Np}$ only requires a simple geometric argument; this will be covered in section \ref{sect:lowToHigh}. We prove the implication $\ref{NewTFAE:N} \Rightarrow \ref{NewTFAE:CH}$ by relying heavily on the work of Schmerl; we will adapt his proof from \cite{ConversenClouds} to higher dimensions. This is done in section \ref{sect:card}.

Finally, we can wonder how many clouds are needed to cover $\R^N$ without any assumptions on the size of the continuum. Komjáth showed that countably many clouds cover $\R^2$; we show with the following theorem that the same holds in $\R^N$. We include the proof of this result in section \ref{sect:lowToHigh}, as the proof follows in the same way as the implication $\ref{NewTFAE:2p} \Rightarrow \ref{NewTFAE:Np}$ mentioned above.

\begin{thm}
\label{thm:countably}
Let $N \geq 2$. Then countably many clouds cover $\R^N$.
\end{thm}

With this last result, we have completely resolved this question as to how many clouds cover $\R^N$, both in relation to the continuum hypothesis and to coverings of $\R^2$. A follow-up question would be to investigate whether these covering results can be extended to $\R^N$ for sets other than clouds. For example, in \cite{ConversenClouds}, Schmerl defines a \textbf{spray} centered at $a \in \R^2$ as a set $C$ such that the intersection of $C$ with any circle (in the usual sense) is finite.

In \cite{Sprays}, following the work of de la Vega in \cite{EquivalenceClasses}, Schmerl shows that for any $n+2$ distinct collinear points in $\R^2$, $2^{\aleph_0} \leq \aleph_n$ if and only if $\R^2$ can be covered by sprays around these points. Schmerl also shows that for any three distinct noncollinear points in $\R^2$, there are sprays centered at these points which cover $\R^2$.

It is natural to ask whether these results can be extended to $\R^N$ for any $N \geq 2$. However, the more general approach would be to investigate the work of de la Vega \cite{EquivalenceClasses}, which determines conditions for what equivalence relations allow finite coverings of $\R^2$. It would be interesting to determine whether his results hold in $\R^N$.

\textbf{Acknowledgment.} The author wishes to extend their gratitude to Professor Marcin Sabok of McGill University. This research was done as an undergraduate research project under his supervision; the author also wishes to thank him for his help in editing this paper.

\section{Lower to Higher Dimensions}
\label{sect:lowToHigh}

The goal of this section is to prove theorem \ref{thm:countably}, as well as the following result, which corresponds to the implication $\ref{NewTFAE:2p} \Rightarrow \ref{NewTFAE:Np}$ of theorem \ref{NewTFAE}.

\begin{thm}
\label{thm:2toN}
Let $n \geq 3$, $N \geq 2$. Suppose that for any $n$ distinct noncollinear points in $\R^2$, there are clouds centered at these points that cover $\R^2$. Then for any $n$ distinct noncollinear points in $\R^N$, there are clouds centered at these points that cover $\R^N$.
\end{thm}

The key idea in the proofs of these theorems is finding a way to extend clouds in $\R^2$ to form clouds in $\R^N$. This is, in fact, very simple: we just extend the cloud linearly into the remaining dimensions. This idea is encapsulated in the following lemma; see also figure \ref{fig:extendCloud}.

\begin{figure}[h]
\includegraphics[width=\textwidth]{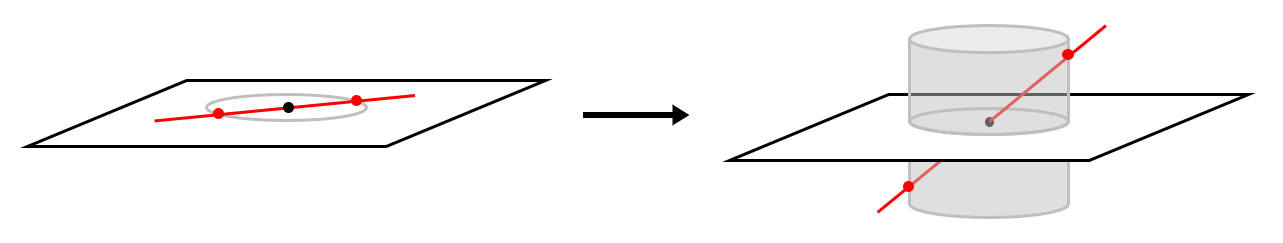}
\centering
\caption{A circle (in grey) is an example of a cloud in $\R^2$ around its origin. Here it is extended into $\R^3$ as described in lemma \ref{lem:cloudExtension}. Lines through the center intersect in "the same" places.}
\label{fig:extendCloud}
\end{figure}

\begin{lem}
\label{lem:cloudExtension}
Let $2 \leq K < N$. Let $C \subset \R^K$ be a cloud centered at $a$. If $a \notin C$, then the set $C' = \{(x,y) \in \R^N : x \in C\}$ is a cloud centered at $a' = (a,b)$ for any $b \in \R^{N-K}$.
\end{lem}

\begin{proof}
We must show that the intersection of $C'$ with any line through $a'$ is finite.
Let $L$ be a line through $a'$. We can write $L$ as $\{a' + tv': t \in \R\}$ for some $v' \neq 0$. Write $v' = (v,w)$ for some $v \in \R^K, w \in \R^{N-K}$. First note that if $v = 0$, then every point in $L$ is of the form $(a,b+tw)$. Since we assumed $a \notin C$, no such points are in $C'$, and so $|L \cap C'| = 0$.

Conversely, suppose $v \neq 0$. Define the map $T : L \rightarrow \R^K : a' + tv' \mapsto a+tv$. Since $v' \neq 0$, this map is well-defined; since $v \neq 0$, it is injective. So, $T(L)$ is a line is $\R^K$ through $a$. Moreover, since $a' + tv' = (a+tv,b+tw)$, we have that $a'+tv' \in C'$ if and only if $a+tv \in C$, by definition of $C'$. Therefore $|L \cap C'| = |T(L) \cap C|$, and since $T(L)$ is a line through $a$ and $C$ is a cloud around $a$, the latter quantity is finite.
\end{proof}

The proof of theorems \ref{thm:countably} and \ref{thm:2toN} are now rather straightforward. For the former, all we have to do is select clouds in $\R^2$ which cover the plane, and extend them into $\R^N$ as in lemma \ref{lem:cloudExtension}. We present this more formally below.

\begin{proof}[Proof of theorem \ref{thm:countably}]
By a result due to Komjáth \cite{ThreeClouds}, we know that $\R^2$ can be covered by countably many clouds (in particular, so-called \textit{circles}). Let $C_1, C_2, ...$ be such clouds centered at $a_1, a_2, ...$. Remark that adding or removing a single point does not change whether a set is a cloud, so we can assume that $a_i \notin C_i$ for each $i$. 

Let $C_1', C_2', ...$ be the extensions to $\R^N$ of these clouds as described in lemma \ref{lem:cloudExtension}. Since $C_i' = \{(x,y) \in \R^N : x \in C_i \}$, it is clear that the $C_i'$ cover $\R^N$, since the $C_i$ cover $\R^2$. Thus we have found countably many clouds that cover $\R^N$.
\end{proof}

The proof of theorem \ref{thm:2toN} is very similar; we project our points in $\R^N$ onto the plane, get clouds around these points, and then extend these clouds into $\R^N$ as in lemma \ref{lem:cloudExtension}. However, we run into a small problem. If we have points which are noncollinear in $\R^N$, their projections may not be noncollinear (or even distinct) in the plane.
To get around this, we first must perform a transformation on $\R^N$ to remove this projection problem. 

This transformation we need must be a \textit{collineation}, i.e. a bijective map which sends collinear points to collinear points. It is easy to see that if $C$ is a cloud centered at $a$ and $f$ is a collineation, then $f(C)$ is a cloud centered at $f(a)$; moreover, $f^{-1}$ is a collineation as well. We will need to find a collineation which can transform our points such that their projections onto $\R^2$ are distinct and noncollinear. This technical result is presented in lemma \ref{lem:restrictionMap} at the end of this section. For now, we present the proof of theorem \ref{thm:2toN} while assuming the existence of such a map.

\begin{proof}[Proof of theorem \ref{thm:2toN}]
In this proof, for any $x \in \R^N$, we'll use $\hat{x} \in \R^2$ to denote the restriction of $x$ to its first two coordinates.
Let $a_1, ..., a_n$ be distinct noncollinear points in $\R^N$. By lemma \ref{lem:restrictionMap}, there exists a collineation $T : \R^N \rightarrow \R^N$ such that $\hat{T(a_1)}, ..., \hat{T(a_n)}$ are distinct noncollinear points in $\R^2$. Then by the assumption of the theorem, there are clouds $\hat{C_1}, ..., \hat{C_n}$ in $\R^2$ centered at $\hat{T(a_1)}, ..., \hat{T(a_n)}$ which cover $\R^2$. 

Since adding or removing a single point does not change whether a set is a cloud, we can assume that $\hat{T(a_i)} \notin \hat{C_i}$ for $1 \leq i \leq n$.
Then we can apply lemma \ref{lem:cloudExtension} to extend these to clouds $C_1, ..., C_n$ centered at $T(a_1), ..., T(a_n)$. Then the $C_i$ cover $\R^N$, since the $\hat{C_i}$ cover $\R^2$.
Finally, we have that $D_i = T^{-1}(C_i)$ is a cloud centered at $a_i$ for each $1 \leq i \leq n$. Since the $C_i$ cover $\R^N$, so do the $D_i$, and we are done.
\end{proof}

Finally, we present the proof of the existence of the collineation we needed.
Note that this result is only needed to ensure that the clouds which cover $\R^N$ are centered at the desired points. If we are just concerned about covering $\R^N$ with clouds centered at any points, we simply extend the clouds to $\R^N$, as is done in the proof of theorem \ref{thm:countably}.

\begin{lem}
\label{lem:restrictionMap}
Let $n \geq 3$, $N \geq 2$. Denote by $\widehat{x}$ the restriction of $x \in \R^N$ to $\R^2$. Let $a_1, ..., a_n \in \R^N$ be distinct and noncollinear. Then there exists a collineation $T : \R^N \rightarrow \R^N$ such that $\widehat{T(a_1)}, ..., \widehat{T(a_n)}$ are distinct noncollinear points in $\R^2$.
\end{lem}

\begin{proof}
The case $N=2$ is trivial, so we assume $N > 2$.
First, define the map $T_1 : x \mapsto x - a_1$, and let $b_i = T_1(a_i)$ for each $i$. Note that $T_1$ is a collineation, and $b_1 = 0$. 

Since the $a_i$ are noncollinear, there exists some $i > 2$ such that $b_i$ does not lie on the line through $b_1 = 0$ and $b_2$ (note $b_2 \neq b_1$ since the $a_i$ are distinct). Without loss of generality, we can say this is $b_3$. Then $b_2$ and $b_3$ are linearly independent, and since $N \geq 2$, we can define a linear bijection $T_2 : \R^N \rightarrow \R^N$ such that $T_2(b_2) = e_1$ and $T_2(b_3) = e_2$. Let $c_i = T_2(b_i)$ for each $i$; note that $T_2$ is a collineation, and $c_1 = 0$, $c_2 = e_1$, $c_3 = e_2$.

Next, for any $\Lambda = (\lambda_3, ..., \lambda_N) \in \R^{N-2}$, let $T_3^{\Lambda}$ be the linear bijection defined by:
\begin{align*}
T_3^{\Lambda}(e_1) = e_1, &&
T_3^{\Lambda}(e_2) = e_2, &&
T_3^{\Lambda}(e_i) = e_i + \lambda_i e_1 \text{ for } i > 2
\end{align*}
Note that for any $\Lambda$, we have $T_3^{\Lambda}(c_1) = 0$, $T_3^{\Lambda}(c_2) = e_1$, $T_3^{\Lambda}(c_3) = e_2$. We claim that there is some $\Lambda \in \R^{N-2}$ such that $\widehat{T_3^{\Lambda}(c_i)} \neq \widehat{T_3^{\Lambda}(c_j)}$ for any $i \neq j$. 

If this is the case, then $T = T_3^{\Lambda} \circ T_2 \circ T_1$ is a collineation. It also maps $a_1, a_2, a_3$ to $0, e_1, e_2$, and the above condition lets us conclude that $\widehat{T(a_1)}, ..., \widehat{T(a_n)}$ are distinct, non-collinear points. So, to conclude the proof, we just have to show that there exists such a $\Lambda$.

Let $i \neq j$, and write $c_i = (x_1^i, ..., x_N^i)$ and $c_j = (x_1^j, ..., x_N^j)$. Then:
\begin{align*}
\widehat{T_3^{\Lambda}(c_i)} = 
\left( x_1^i + \sum \limits_{k=3}^N \lambda_k x_k^i 
\quad , \quad x_2^i \right)
&&
\widehat{T_3^{\Lambda}(c_j)} = 
\left( x_1^j + \sum \limits_{k=3}^N \lambda_k x_k^j 
\quad , \quad x_2^j \right)
\end{align*}
For these to be equal, we must have:
\begin{align*}
x_2^i - x_2^j = 0
&& \text{and} &&
(x_1^i - x_1^j) + \sum \limits_{k=3}^N \lambda_k (x_k^i - x_k^j) = 0
\end{align*}
If this is the case, then we cannot have that $x_k^i = x_k^j$ for every $k \geq 3$; otherwise we'd need $x_1^i = x_1^j$ and $x_2^i = x_2^j$ as well, which would mean $c_i = c_j$. Since $T_1$ and $T_2$ are bijections and the $a_i$ are distinct, the $c_i$ must be distinct, so this is a contradiction. Therefore, we have one non-trivial condition on the values of the $\lambda_k$, so
 the set of $\Lambda \in \R^{N-2}$ that make these equalities hold is a hyperplane.

That is, for each $1 \leq i < j \leq n$, the set of $\Lambda$ such that $\widehat{T_3^{\Lambda}(c_i)} = \widehat{T_3^{\Lambda}(c_j)}$ is a set of measure zero in $\R^{N-2}$. Therefore, we can pick some $\Lambda$ such that this equality does not hold for every $1 \leq i < j \leq n$. This concludes the proof.
\end{proof}

\section{Existence of Clouds Implies Cardinality Statements}
\label{sect:card}

The goal of this section is to prove the following result, which corresponds to the implication $\ref{NewTFAE:N} \Rightarrow \ref{NewTFAE:CH}$ of theorem \ref{NewTFAE}. As noted in the introduction, the case $N=2$ was proved by Schmerl \cite{ConversenClouds}.

\begin{thm}
\label{thm:schmerlAdapt}
Let $1 \leq n < \omega$, $N \geq 2$. If $\R^N$ can be covered by $n+2$ clouds, then $2^{\aleph_0} \leq \aleph_n$.
\end{thm}

This proof will follow the ideas of Schmerl and make use of the projective space $P_m(\R)$. This corresponds to the equivalence classes of $\R^{m+1} \backslash \{0\}$ with respect to the relation $\twiddle$, where $x \twiddle y$ if they are on the same line through the origin. We denote the equivalence class of $x$ by $[x]$; we call these the \textit{homogeneous coordinates} of $x$. We also have \textit{lines} in $P_m(\R)$; these are the images of planes through the origin under the map $x \mapsto [x]$. We will need the following lemmas regarding $P_m(\R)$ for the proof.

\begin{lem}
\label{lem:embedding}
For any $m \geq 1$, the map $E : \R^m \rightarrow P_m(\R) : x \mapsto [x,1]$ is an injective, continuous, open map. Moreover, for any line $L \subset \R^m$, there is a point $p \in P_m(\R)$ such that $E(L) \cup \{p\}$ is a line in $P_m(\R)$. We call this the \textbf{point at infinity} of $L$.
\end{lem}

\begin{lem}
\label{lem:collineation}
Let $m \geq 1$, and let $\{x_1, ..., x_{m+1}\}$, $\{y_1, ..., y_{m+1}\}$ be linearly independent subsets of $\R^{m+1}$. Then there exists a homeomorphic collineation $S : P_m(\R) \rightarrow P_m(\R)$ such that $S([x_i]) = [y_i]$ for every $1 \leq i \leq m+1$.
\end{lem}

Finally, the core of Schmerl's proof is to make use of the following theorem, which he attributes to Kuratowski.

\begin{thm}[Kuratowski]
Suppose that $n < \omega$ and that $X$ is any set. Then $|X| \leq \aleph_n$ if and only if there exist $D_1, ..., D_{n+2} \subset X^{n+2}$ which cover $X^{n+2}$, such that $D_i \cap \ell$ is finite whenever $1 \leq i \leq n+2$ and $\ell \subset X^{n+2}$ is a line parallel to the $i^{th}$ coordinate axis.
\end{thm}

Note that in this theorem, when we say a line $\ell$ parallel to the $i^{th}$ coordinate axis, we mean a set of the following form:
\begin{align*}
\ell =  \{ (a_1, ..., a_{i-1}, x, a_{i+1}, ..., a_{n+2}) \in X^{n+2} : x \in X \}
\end{align*}
for some $a_1, ..., a_{i-1}, a_{i+1}, ..., a_{n+2} \in X$. We are now ready to prove the result we want.

\textit{Proof of theorem \ref{thm:schmerlAdapt}.}
Let $C_1, ..., C_{n+2}$ be clouds which cover $\R^N$, centered at points $p_1, ..., p_{n+2}$, respectively. We can assume without loss of generality that $p_i \neq 0$ for each $1 \leq i \leq n+2$, since a translation still results in a covering of $\R^N$ by clouds.

The key of this proof is noting the similarity between the characteristics of our clouds and the sets $D_i$ from Kuratowski's theorem. The only difference is that $C_i \cap \ell$ is finite for lines $\ell$ through $p_i$, whereas $D_i \cap \ell$ is finite for lines $\ell$ parallel to the $i^{th}$ coordinate axis. The idea, then, is to apply some transformation to turn lines through $p_i$ into parallel lines; this is done by moving the points $p_i$ to points at infinity in projective space!

However, we will first need to move our clouds into $\R^{n+2}$, and define some maps, before we can perform this transformation. First, let $T : \R^{n+2} \rightarrow \R^N$ be the linear map such that $T(e_i) = p_i$ for each $1 \leq i \leq n+2$. Next, let $E : \R^{n+2} \rightarrow P_{n+2}(\R)$ be the embedding described in lemma \ref{lem:embedding}, and let $\infty_i$ denote the point at infinity on the $i^{th}$ coordinate axis (again, as in lemma \ref{lem:embedding}). Finally, by lemma \ref{lem:collineation}, there exists a homeomorphic collineation $S: P_{n+2}(\R) \rightarrow P_{n+2}(\R)$ such that $S(E(0)) = E(0)$ and $S(E(e_i)) = \infty_i$ for each $i$.

Now, the need to first move from $\R^N$ to $\R^{n+2}$ means we will need to be careful about where Kuratowski's theorem may hold. So, let us assume that there is some interval $X = (-\epsilon, \epsilon) \subset \R$ with the following two properties (see figure \ref{fig:functions} for illustration):
\begin{enumerate}
\item \label{cond:Q}
$S^{-1}(E(X^{n+2})) \subset E(\R^{n+2})$
\item \label{cond:injective}
If $N = E^{-1}(S^{-1}(E(X^{n+2})))$, then $T$ is injective on lines which pass through some $e_i$ and some $x \in N$.
\end{enumerate}

\begin{wrapfigure}[37]{R}{5cm}
\begin{center}
\includegraphics[width=4.25cm]{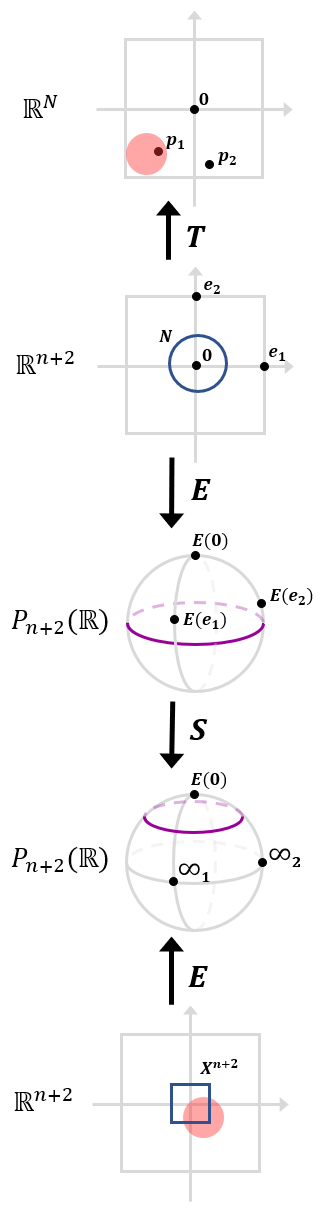}
\end{center}
\caption{Illustration of the mappings $T$, $E$, and $S$ (see the footnote).}
\label{fig:functions}
\end{wrapfigure}

Then we can finally transform our clouds by letting $D_i = X^{n+2} \cap E^{-1}(S(E(T^{-1}(C_i)))$ (again, see figure \ref{fig:functions}). We claim in fact that these $X$ and $D_i$ satisfy the conditions of Kuratowski's theorem.
\blfootnote{In figure \ref{fig:functions}, we take $n=0$, $N=2$ for illustration, even if we assume $n\geq 1$. The sphere represents $P_2(\R)$, since it is a sphere where antipodal points are identified.}
\blfootnote{\textit{In red}: a cloud around $p_1$ in $\R^N$ is transformed via the maps $T$, $E$, $S$, and $E$ again. We move to $P_{n+2}(\R)$, move $E(e_i)$ to infinity, and go back.
}
\blfootnote{\textit{In blue}: the set $X^{n+2}$, when moved back through projective space, is the set $N$; we require $T$ to be injective on lines through $N$ and $e_i$.
}
\blfootnote{\textit{In purple}: the equator of the sphere represents the points of $P_{n+2}(\R)$ not in the image of $E$; these are moved by $S$, but leave an open set around $E(0)$. We require $E(X^{n+2})$ to not intersect the purple set, so that it lands in the image of $E$ after taking $S^{-1}$.
}

We check this claim. Let $1 \leq i \leq n+2$ and let $\ell \subset X^{n+2}$ be a line parallel to the $i^{th}$ coordinate axis; we claim $D_i \cap \ell$ is finite. First, remark that since $E$ is injective and $S$ is bijective, we have:
\begin{align*}
|D_i \cap \ell| = |E(D_i \cap \ell)| = |S^{-1}(E(D_i \cap \ell))|
\end{align*}
Then, since the image of $E$ contains $S^{-1}(E(X^{n+2}))$ by assumption \ref{cond:Q}, we have:
\begin{align*}
|S^{-1}(E(D_i \cap \ell))|
\leq
|E^{-1}(S^{-1}(E(D_i \cap \ell)))|
\end{align*}
Combining this with the fact that $E$, $S^{-1}$, $E^{-1}$ are injective functions on their domains yields:
\begin{align*}
|D_i \cap \ell| \leq | E^{-1}(S^{-1}(E(D_i))) \cap E^{-1}(S^{-1}(E(\ell)))|
\end{align*}

Now, since $E$ and $S$ are collineations, we have that the set $E^{-1}(S^{-1}(E(\ell)))$ is contained within a line $L \subset \R^{n+2}$. Note that $L$ passes through the set $N$ described in assumption \ref{cond:injective}, since $\ell \subset X^{n+2}$. Moreover, since $\ell$ is parallel to the $i^{th}$ coordinate axis, $E(\ell)$ is contained in a line passing through $\infty_i$. Thus the line containing $S^{-1}(E(\ell))$ passes through $E(e_i)$, so $L$ passes through $e_i$. (Note: this is the transformation of parallel lines into lines through points!) By assumption \ref{cond:injective}, we obtain that $T$ is injective on $L$, and so we find that:
\begin{align*}
|D_i \cap \ell| \leq |E^{-1}(S^{-1}(E(D_i))) \cap L| 
\\= |T(E^{-1}(S^{-1}(E(D_i)))) \cap T(L)| 
\\ \leq |C_i \cap T(L)|
\end{align*}
Since $L$ is a line through $e_i$ and $T$ is injective on $L$, $T(L)$ is a line through $p_i$. We conclude that the last expression is finite, since $C_i$ is a cloud around $p_i$. Thus $|D_i \cap \ell|$ is finite, so the condition of Kuratowski's theorem is satisfied.

Since we have found sets $X$ and $D_i$ such that the condition of Kuratowski's theorem is satisfied, we conclude that $\aleph_n \geq |X| = |(-\epsilon, \epsilon)| = 2^{\aleph_0}$; this is what we wanted to show. All we need to do now is to show the existence of an interval $X$ which satisfies the assumptions we gave.

First, since $T$ is continuous and $T(e_i) = p_i \neq 0$ for each $i$, we know that there is some $\delta > 0$ such that $0 \notin T(B(e_i, \delta))$ for each $1 \leq i \leq n+2$. We claim that $T$ is injective on lines $L \subset \R^{n+2}$ which pass through some $e_i$ and some $x \in B(0,\delta)$. Indeed, we can write any such $L$ as $L = \{ \alpha(e_i - x) + e_i : \alpha \in \R\}$. Note that $e_i-x \in B(e_i, \delta)$, so $T(e_i-x) \neq 0$. Therefore, we have the following implication for any $\alpha_1, \alpha_2 \in \R$.
\begin{align*}
& 
T(\alpha_1 (e_i-x) + e_i) = T(\alpha_2 (e_i-x) + e_i)
\\ \Rightarrow \qquad &
\alpha_1 T(e_i-x) + T(e_i) = \alpha_2 T(e_i-x) + T(e_i)
\\ \Rightarrow \qquad & 
\alpha_1 = \alpha_2
\end{align*}
So, $T$ is indeed injective on $L$.

Next, let $Q = P_{n+2}(\R) \backslash E(\R^{n+2})$ (see the purple band in figure \ref{fig:functions}). Since $E$ is open, $Q$ is closed; since $S$ is a homeomorphism, $S(Q)$ is closed. Then since $S(E(0)) = E(0)$, there exists some open set $R_1 \subset P_{n+2}(\R)$ such that $E(0) \in R_1$ and $R_1 \cap S(Q) = \varnothing$.

We now let $R_2 = R_1 \cap S(E(B(0,\delta)))$. Since $S$ and $E$ are open, $R_2$ is open; since $S(E(0)) = E(0)$, we have $E(0) \in R_2$. Finally, we let $R_3 = E^{-1}(R_2) \subset \R^{n+2}$. Since $E$ is continuous, $R_3$ is open; since $E(0) \in R_2$, we get $0 \in R_3$. Thus we can pick $X$ to be an interval $(-\epsilon, \epsilon)$ such that $X^{n+2} \subset R_3$. We claim that $X$ has the desired properties.

First, since $X^{n+2} \subset R_3$, we have $E(X^{n+2}) \subset R_2 \subset R_1$, so $E(X^{n+2}) \cap S(Q) = \varnothing$. By bijectivity of $S$, we get $S^{-1}(E(X^{n+2})) \cap Q = \varnothing$, so $S^{-1}(E(X^{n+2})) \subset E(\R^{n+2})$, as desired.

Next, we similarly have $E(X^{n+2}) \subset R_2 \subset S(E(B(0,\delta)))$, so injectivity of $S$ and $E$ tells us that $E^{-1}(S^{-1}(E(X^{n+2}))) \subset B(0,\delta)$. From what we showed above, we conclude that $T$ is injective on lines through some $e_i$ and $E^{-1}(S^{-1}(E(X^{n+2})))$. Thus we have shown that there exists an $X$ which satisfies all the desired properties, and so this ends the proof.
\hspace*{\fill} \qed

\end{document}